\documentclass[11pt]{amsart}

\usepackage{epsf,amssymb,amsmath,overpic,amscd,psfrag,stmaryrd,times,mycommands,mathrsfs, euscript,thmtools}

\usepackage[style=alphabetic,backend=biber]{biblatex}

\usepackage[all]{xy}
\usepackage{hyperref, stmaryrd}
\usepackage[dvipsnames]{xcolor}
\usepackage{enumitem}
\usepackage{tikz,tikz-cd}

\input{biblatex.cfg}

\makeatletter
\newcommand{\labitem}[2]{%
\def\@itemlabel{#1}
\item
\def\@currentlabel{#1}\label{#2}}
\makeatother

\begin{document}
\title[Coherent orientations in symplectic field theory revisited]
{Coherent orientations in symplectic field theory revisited}

\author{Erkao Bao}
\address{School of Mathematics, University of Minnesota, Minneapolis, MN 55455}
\email{bao@umn.edu} \urladdr{https://erkaobao.github.io/math/}



\begin{abstract}
In symplectic field theory (SFT), the moduli spaces of $J$-holomorphic curves can be oriented coherently (compatible with gluing). 
In this note, we correct the signs involved in the generating function $\Ha$ in SFT so that the master equation $\Ha \cdot \Ha = 0$ holds assuming transversality. The orientation convention that we use is consistent with that of Hutchings-Taubes from \cite{hutchings2009gluing}, but differs from that of Bourgeois-Mohnke in \cite{bourgeois2004coherent}. 
\end{abstract}

\maketitle

\setcounter{tocdepth}{2}


\section{Introduction}

Symplectic Field Theory (SFT) was introduced by Eliashberg, Hofer and Givental in \cite{eliashberg2000introduction}, and is a generalization of the Gromov-Witten invariants in the spirit of a topological field theory. SFT packs a signed count of elements of moduli spaces of $J$-holomorphic curves in the symplectization of a contact manifold into a potential function $\Ha$. The potential function $\Ha$ satisfies the master equation $\Ha \cdot \Ha = 0$. There are different choices to orient the moduli spaces of $J$-holomorphic curves, such as those of Bourgeois-Mohnke \cite{bourgeois2004coherent} and Hutchings-Taubes \cite{hutchings2009gluing}.
However, the algebraic setup of $\Ha$ in \cite{eliashberg2000introduction} is not compatible with either of the two orientation conventions.

In this paper, correct the signs of $\Ha$ so that they match with the orientation convention  of \cite{hutchings2009gluing}\footnote{It is likely that a different correction of $\Ha$ can match the orientation convention of \cite{bourgeois2002morse} resulting in an equivalent SFT.}; 


In Section 2, we recall the coherent orientations for Cauchy-Riemann tuples used in \cite{hutchings2009gluing}. In Section 3, we orient moduli spaces of $J$-holomorphic curves. 
In Section 4, we correct the definition of $\Ha$ and prove $\Ha \cdot \Ha = 0$.

\section{Coherent orientations of Cauchy-Riemann tuples}

\subsection{Cauchy-Riemann tuples}

\begin{definition}
A {\em decorated Riemann surface} with $(k_+, k_-)$ marked points is a tuple $(\Sigma, j, \bs p, \bs r)$ such that
\be 
    \item $(\Sigma, j)$ is a {\em possibly disconnected} closed Riemann surface, 
    \item $\bs p = (\bs p^+, \bs p^-)$ and $\bs p^\pm = (p^\pm _1,\dots, p^\pm_{k_\pm})$ is an ordered tuple of points on $\Sigma$, 
    \item $\bs r = (\bs r^+, \bs r^-)$ and $\bs r^\pm = (r^\pm_1, \dots, r^\pm_{k_\pm})$ is an ordered  tuple of rays on $\Sigma$ at $\bs p$, i.e., $r^\pm_i \in T_{p^\pm_i} \Sigma - \{0\}$.
\ee 
\end{definition}

Given a decorated Riemann surface $(\Sigma, j, \bs p, \bs r)$,
we regard its marked points as punctures and find holomorphic cylindrical coordinates around them.
Let $\phi_i^\pm: D\subset \C \to \mathcal U_{p_i^\pm}\subset \Sigma$ be a biholomorphic map from the unit disc $D$ to a neighborhood $\mathcal U_{p_i^\pm}$ of $p^\pm_i$ such that $\phi_i^\pm(o) = p_i^\pm$ and $d\phi_i^\pm(\frac{\partial}{\partial x}) = r^\pm_i,$ where $o\in D$ is the origin and $\frac{\partial}{\partial x} \in T_o D.$ 
Let $\dot{\mathcal U}_{p_i^\pm} = {\mathcal U}_{p_i^\pm} - \{p_i^\pm\},$
and $h_i^\pm:\R^{\geq 0 (\leq 0)} \times S^1 \to \dot{\mathcal U}_{p_i^\pm}$ be the biholomorphic map defined by 
\begin{equation} \label{eqn: h_i}
    h_i^\pm (s,t) =(\phi^\pm_i)^{-1}(e^{\mp s \mp \sqrt{-1} t}).
\end{equation}

\begin{definition}[\cite{schwarz1995cohomology}]
A smooth loop of symmetric matrices $$S\in C^\infty(S^1, \op{End}(\R^{2n-2}))$$ is called {\em admissible} if the ordinary differential equation $$\dot x(t) = J_0  S(t) x(t),$$ $$x:S^1 \to \R^{2n-2}$$ has only the zero solution, where $J_0$ is the standard complex structure on $\R^{2n-2}$.
\end{definition}

\begin{definition}
A {\em Cauchy-Riemann tuple} (CR tuple for short) is a tuple $$\T = (\Sigma, j, \bs p, \bs r, E, J, \bs \psi =\{\psi_i^\pm\}_i, \bs S = \{S_i^\pm\}_i)$$ consisting of the following data:
\be
    \item A decorated Riemann surface $(\Sigma, j, \bs p, \bs r)$ with $(k_+,k_-)$ marked points. 
    \item A complex vector bundle $(E, J)$  over $\dot \Sigma = \Sigma - \bs p$, such that for each $p_i^\pm$, there exist: 
    \be 
    \item  A neighborhood $\mathcal U_{p_i^\pm}$ of $p_i^\pm$ in $\Sigma$ and a trivialization 
    \begin{equation}\label{formula: trivialization of E}
    \psi_i^\pm: (E, J)|_{\dot{\mathcal U}_{p_i^\pm}} \simeq (\R^{2n} \times \dot{\mathcal U}_{p_i^\pm}, J_0),
    \end{equation} 
    where $\dot{\mathcal U}_{p_i^\pm} = \mathcal U_{p_i^\pm} - p_i^\pm$, and $J_0$ is the standard complex structure on $\R^{2n}$.
    \item An admissible $S_i\in C^\infty(S^1, \op{End}(\R^{2n-2}))$, where $\R^{2n-2}$ is viewed as the last $(2n-2)$ factors of $\R^{2n}$. 
    \ee 
\ee 
\end{definition}


Given an admissible loop $S$, we obtain a path of symplectic matrices $B(t) \in \op{Symp}(2n-2, \R)$ that solves $\dot B(t) = J_0 S(t)B(t)$ and $B(0) = \op{Id}.$ We define the Conley-Zenhder index of $S$ by $\mu_{\op{CZ}}(S)= \mu(\{B(t)\}_t)$, where $\mu$ is the Maslov index, and we grade $S$ over $\Z_2$ by $|S| = \mu_{\op{CZ}}(S) + (n-1) \mod 2.$

\begin{lemma}[\cite{bourgeois2004coherent}]
Given an admissible loop $S\in C^\infty (S^1, \op{End}(\R^{2n-2}))$, the associated operator $A = J_0\frac{\partial}{\partial t} + S: W^{1,p}(S^1, \R^{2n-2}) \to L^p(S^1, \R^{2n-2})$ has discrete spectrum $\sigma(A) \subset \R -\{0\}.$
\end{lemma}

For an admissible loop $S$, 
we define $\lambda_S = \min\{-\lambda_{-1}, \lambda_1\}>0$, where $\lambda_1$ is the smallest positive eigenvalue of $J_0\frac{\partial}{\partial t} + S$, and  $\lambda_{-1}$ is the largest negative eigenvalue of $J_0\frac{\partial}{\partial t} + S$.

Fix $k\in \Z^{\geq 0}$, and $p> 1$ such that $kp > 2.$
\begin{definition}[Cauchy-Riemann operators]\label{def: CR operators}
For a CR tuple $\T$, we define $\CRO(\T)$ to be the topological space of linear operators, called {\em Cauchy-Riemann operators} (CR operators for short):
\begin{equation}\label{formula: Fredholm operator}
L: W^{k,p}_\delta (\dot \Sigma, E) \oplus \mathcal V \to W^{k-1,p}_\delta (\dot \Sigma, \wedge^{0,1}T^*\dot\Sigma \otimes_\C E ),
\end{equation}
such that: 
\be 
    \item $0< \delta < \min_i \lambda_{S_i}$, and $W^{k,p}_\delta(\cdot)$ and $W^{k-1,p}_\delta(\cdot)$ are weighted Sobolev spaces.
    \item $\mathcal V = (\oplus_{i=1}^{k_-} \mathcal V_i^-) \oplus (\oplus_{i=1}^{k_+} \mathcal V_i^+)$, where $\mathcal V^\pm_i = \op{span}\{\beta_i^\pm\cdot e_1, \beta_i^\pm\cdot e_{2}\}$,  $\{e_1,\dots, e_{2n}\}$ is the standard basis of $\R^2 \oplus \R^{2n-2}$, $0\leq \beta_i^\pm \leq 1$ is a bump function that is supported in $\mathcal U_{p_i^\pm}$ and satisfies  $\beta_i^\pm(p_i^\pm) = 1$, and hence $\beta_i^\pm\cdot e_1$ and $\beta_i^\pm\cdot e_2$ are sections of $E$.
    \item $L$ is a real Cauchy-Riemann operator (Appendix C.1 in \cite{mcduff2012j}), i.e., $L = L_0 + L_1$, 
    where $L_0\in W^{k-1,p}(\dot \Sigma, \wedge^{0,1} T^*\dot \Sigma \otimes_\R \op{End}_\R (E))$ is the $0$-th order operator, 
    and $L_1$ is a complex Cauchy-Riemann operator, that is, 
    $L_1$ is a $\C$-linear operator that satisfies the Leibnitz rule
    $$L_1 (f \eta) = f(L_1 \eta) + (\overline \partial f) \eta$$ for any $f\in C^\infty_c (\dot \Sigma)$ and $\eta \in W^{k,p}_\delta (\dot \Sigma, E) \oplus \mathcal V$. 
    \item With respect to the coordinate $h_i^\pm$ and the trivialization $\psi_i^\pm$, we require that for any $\eta \in W^{k,p}_\delta (\dot \Sigma, E) \oplus \mathcal V$ with support in ${\dot{\mathcal U}_{p_i^\pm}}$:
    \begin{equation}\label{eqn: CR operator near ends}
        L\eta(s,t) = \left(\frac{\partial \eta}{\partial s} + \widetilde J_i^\pm (s,t) \frac{\partial \eta}{\partial t} + \widetilde S_i^\pm (s,t) \eta \right)\otimes (ds - \sqrt{-1} dt),
    \end{equation}
    where $(s,t)$ is the cylindrical coordinates around $p_i^\pm$, 
    $\widetilde J_i^\pm (s,t)$ is a complex structure on $\R^{2n}$, $\widetilde S_i^\pm (s,t)\in \op{End}(\R^{2n})$, and there exist some constants $C_i^\pm > 0$ such that for all $\beta = (\beta_1,\beta_2) \in \Z^{\geq 0} \times \Z^{\geq 0}$ with $\beta_1 +\beta_2 \leq k$ one has 
    $$|\partial^\beta (\widetilde J_i^\pm (s,t)- \widehat J_0)| \leq C_i^{\pm} e^{\mp \frac{1}{2}\lambda_{S_i^\pm} s},$$  
    $$|\partial^\beta (\widetilde S_i^\pm (s,t)-\widehat S_{i}^\pm (t))| \leq C_i^{\pm} e^{\mp \frac{1}{2}\lambda_{S_i^\pm} s},$$
    where
    $\widehat J_0$ is the standard complex structure on $\R^{2n}$,
    $$
       \widehat S_i^\pm =  
       \begin{pmatrix}
       0 &  & \\
     & 0 & & \\
    & &  S_i^\pm
       \end{pmatrix}
    $$ and $\partial^\beta = \partial^{\beta_1}_s \partial^{\beta_2}_t.$ 
\ee 

\end{definition}

Note that the third term $\widetilde S_i^\pm (s,t) \eta \otimes (ds - \sqrt{-1} dt)$ on the right hand side of Formula~\eqref{eqn: CR operator near ends} is the $0$-th order operator as in (3). 


\begin{proposition}[\cite{bourgeois2002morse, bourgeois2004coherent}]
Any $L \in \CRO(\T)$ is Fredholm, and its Fredholm index is given by $$\op{ind} L = \sum_{i=1}^{k_+} \mu_{\op{CZ}}(S^+_i) - \sum_{i=1}^{k_-} \mu_{\op{CZ}}(S^-_i) - (n-1)(k_- + k_+) + 2c_1(E) + n (2-2g),$$
where $c_1(E)$ is the relative $1$-st Chern number of $E$ with respect to the trivialization $\bs \psi$. 
\end{proposition}

We define the $\op{mod} 2$ indices $$\op{ind}^\pm \T := \sum_{i=1}^{k_\pm} |S^\pm_i| \mod 2,$$ 
and $$\op{ind}\T :=\op{ind}^+ \T + \op{ind}^- \T = \op{ind}L \mod 2.$$




\begin{definition}[determinant line]
    For a CR operator $L$, we define its {\em determinant line} as 
 $$\det L = \wedge ^{\op{top}} \ker L \otimes \wedge^{\op{top}} (\op{coker} L)^*.$$ 
\end{definition}
  
\begin{definition}[orientation]
For a CR operator $L$, we define its {\em orientation} $o(L)$ to be a choice of a non-zero vector in $\det L$ up to positive scalar multiplication.
\end{definition}

\begin{example}\label{example: trivial CR tuple}
We call a CR tuple $\T$ {\em trivial}, if $k_+ = k_- = 1$ and $S_1^+ = S_1^-$.
For a trivial $\T$, $L \in \CRO(\T)$ is said to be {\em trivial}, if $J_1^+ = J_1^-$, and $\widetilde J^\pm_1$ and $\widetilde S^\pm_1$ are independent of $s$.
In this case, $\op{ind} L = 2$, $\op{coker} L = \{0\},$ and $\ker L$ is spanned by translation and rotation.
\end{example}

\begin{example}
When $k_- = 0 = k_+$, the operator $L$ is homotopic to a complex CR operator, whose $\ker$ and  $\op{coker}$ are complex vector spaces. We have the canonical orientation $\ocan(L)$ of $\det L$ coming from the complex structure.
\end{example}

It is convenient to state the following lemma:

\begin{lemma}[\cite{fooo} p. 676]\label{lemma: linear algebra}
Let $V$ and $W$ be Banach spaces, and $\phi:V\to W$ a linear Fredholm operator. 
Let $F \subset W$ be a finite-dimensional subspace of $W$ such that $W = \op{im}(\phi) + F$.
Then there is an isomorphism $$\det \phi \simeq  \wedge^{\op{top}} (\phi^{-1}(F)) \otimes \wedge^{\op{top}} F^*,$$
which is natural up to a positive constant.
More precisely, suppose $\phi^{-1}(F) = \ker \phi \oplus H \subseteq V$, $\{e_1, \dots, e_n\}$ is a basis of $\ker \phi$, $\{h_1, \dots, h_m\}$ is a basis of $H$, and $\{v_{1},\dots, v_\ell, \phi(h_1), \dots, \phi(h_m)\}$ is a basis of $F$.
Then the isomorphism is given by
\begin{align*}
e_1 \wedge \dots \wedge e_n  \otimes v_\ell ^* \wedge \dots \wedge v_1^* \mapsto &
e_1\wedge \dots \wedge e_n \wedge h_1\wedge \dots \wedge h_m \\
& \otimes \phi(h_m)^* \wedge \dots \wedge \phi(h_1)^* \wedge v_\ell ^* \wedge \dots \wedge v_1^* .
\end{align*}
\end{lemma}

Given a continuous family of Fredholm operators $\{\phi_{\tau}\}_{\tau\in [0,1]}$, we can find a finite-dimensional space $F \subset  W$ and a subspace $U \subset W$ such that $W = F \oplus U$ and $W = \op{im}(\phi_\tau) + F$. Let $\pi_{U}: W \to U$ be the projection map.
Then the map $\pi_U \circ \phi_\tau: V \to U$ is surjective, and $\ker (\pi_U \circ \phi_\tau) = \phi_\tau ^{-1}(F).$
By Lemma~\ref{lemma: linear algebra}, we have $\det \phi_\tau \simeq \wedge^{\op{top}}(\phi_\tau^{-1}(F))\otimes \wedge ^{\op{top}} F^* \simeq \wedge^{\op{top}} \ker (\pi_U \circ \phi_\tau) \otimes \wedge^{\op{top}} F^*$. Since both $\ker (\pi_U \circ \phi_\tau)$ and $F^*$ form vector bundles over $[0,1]$, $\det \phi_\tau$ forms a line bundle over $[0,1]$.
More generally, a homotopy of CR tuples $\{L_\tau\}_{0\leq \tau \leq 1}$ induces an isomorphism $\det L_0 \simeq \det L_1$.

Since the space $\CRO(\T)$ is contractible, $\det L$ and $\det L'$ are canonically isomorphic for any $L, L' \in \CRO(\T)$. For this reason, we also write $\det \T$, $o(\T)$, and $\ocan(\T)$.


\subsection{Disjoint union}
The disjoint union $\T \coprod \T'$ of two CR tuples $\T$ and $\T'$ is defined in the obvious way. The punctures of $\dot \Sigma \coprod \dot \Sigma'$ are ordered so that the punctures of $\dot\Sigma$ come before those of $\dot\Sigma'$, and the relative orders of the punctures of $\dot \Sigma$ and $\dot \Sigma'$ are preserved, respectively.
For any $L \in \CRO (\T)$ and $L' \in \CRO(\T')$, the disjoint union map induces an isomorphism 
\begin{equation} \label{formula: L disjoint L'}
    \det L \otimes \det L' \to \det (L \coprod L'),
\end{equation}
where $L \coprod L' \in \CRO(\T \coprod \T')$ is defined as follows:
Let $\{e_1, \dots, e_n\}$ be a basis of $\ker L$, $\{f_1, \dots, f_m\}$ be a basis of $\op{coker} L$,
$\{e'_1, \dots, e'_{n'}\}$ be a basis of $\ker L'$, $\{f'_1, \dots, f'_{m'}\}$ be a basis of $\op{coker} L'$.
Then the isomorphism is given by:
 
\begin{equation}\label{eqn: disjoint union basis expression}
\begin{split}
   & e_1\wedge\dots \wedge e_n \otimes f_m^* \wedge \dots \wedge f_1^*  \otimes e'_1\wedge \dots \wedge e'_{n'}\otimes {f'_{m'}}^* \wedge \dots \wedge {f'_1}^* \\
    \mapsto & (-1)^{\op{ind} L' \cdot \dim(\op{coker}L)} e_1\wedge\dots \wedge e_n \wedge  e'_1\wedge \dots \wedge e'_{n'} \\
   & \otimes {f'_{m'}}^* \wedge \dots \wedge {f'_1}^* \wedge f_m^* \wedge \dots \wedge f_1^*. 
\end{split} 
\end{equation}

In the case when both $L$ and $L'$ are surjective, $L \coprod L'$ is also surjective, and the  isomorphism simplifies to 
$$e_1 \wedge \dots \wedge e_n \otimes e'_1 \wedge \dots \wedge e'_{n'} \mapsto e_1 \wedge \dots \wedge e_n \wedge e'_1 \wedge \dots \wedge e'_{n'}.$$

\begin{remark}
    This sign $(-1)^{\op{ind} L' \cdot \dim(\op{coker}L)}$ is the same sign that comes from ``passing" the term $f_m^* \wedge \dots \wedge f_1^*$ in the right hand side of Formula~\eqref{eqn: disjoint union basis expression} across the term $e'_1\wedge \dots \wedge e'_{n'}  \otimes {f'_{m'}}^* \wedge \dots \wedge {f'_1}^*$.
\end{remark} 

The isomorphism \eqref{formula: L disjoint L'} is continuous with respect to the homotopy of CR operators $L$ and $L'$,
hence induces a map on CR tuples $\T$ and $\T'$:
\[
     \det \T \otimes \det \T'   \to \det (\T \coprod \T') 
\]
and we denote the image of $v\otimes v'$ by $v\coprod v'$.

Note that $\det (\T \coprod \T')$ can be canonically identified with $\det (\T' \coprod \T)$ by identifying  the disconnected Riemann surface $\Sigma \coprod \Sigma'$
with $\Sigma' \coprod \Sigma$, and identifying the bundles $E\coprod E'$ with $E' \coprod E$ in the obvious way.
Under such identification, 
for any $v\in \det \T$ and $v' \in \det \T'$,
$v \coprod v'$ and $v' \coprod v$ lie in the same vector space.
The following lemma is clear from the above definition.

\begin{lemma} \label{lemma: disjoint union}
$$v \coprod v' = (-1)^{\op{ind}\T  \op{ind}\T'} v' \coprod v,$$
for any $v\in \det \T$ and $v' \in \det \T'$.
\end{lemma}

\subsection{Gluing} \label{section: gluing}
Given two CR tuples $\T$ and $\T'$, and $L \in \CRO(\T)$ and $L' \in \CRO(\T')$,
a positive integer $\tau \leq \op{min}(k_-, k_+')$, and a sufficiently large gluing parameter $R > 0$, if the first $\tau$ negative ends of $\dot \Sigma$ match the last $\tau$ positive ends of $\dot \Sigma'$, i.e.,
for $1\leq i\leq \tau$, 
\[S_i^- = {S'}_{k'_+ - \tau +i}^{+},\]
then we can glue $\T$ and $\T'$ to obtain a new CR tuple $\T'',$
and glue $L$ and $L'$ to obtain $L'' \in \CRO (\T'')$.
The gluing construction is straightforward: for each $1\leq i \leq \tau$ we ``chop off" the end $(-\infty, -2R]\times S^1$ from $\dot \Sigma$ around $p_i^-$, and the end $[2R, \infty) \times S^1$ from $\dot \Sigma'$ around $p^+_{k'_+ -\tau + i}$ , identify the regions $[-2R, -R]\times S^1 \subset \dot \Sigma$ and $[R, 2R]\times S^1\subset \dot \Sigma'$, and the bundles $E$ and $E'$ above these regions, and interpolate $L$ and $L'$ over the identified regions.
We now explain how $\bs p''$ is ordered. 
The positive (negative) marked points of $\bs p'$ are ordered before the positive (negative) marked points of $\bs p$, more precisely,
$\bs p''_+ = ({p'}_+^1, \dots, {p'}_+^{k'_+ - \tau}, {p}_+^{1}, \dots, {p}_+^{k_+})$ and $\bs p_-'' = ({p'}_-^1, \dots, {p'}_-^{k'_-}, p_-^{\tau + 1}, \dots, p_-^{k_-})$. We denote $\T'' = \T \sharp_{\tau, R} \T'$ and $L'' = L \sharp_{\tau, R} L'$.

We follow the {\em complete gluing} convention in \cite{hutchings2009gluing} by restricting the gluing to the case when $k_- = k_+' = \tau$, i.e., all the negative punctures of $\dot \Sigma$ match with all the positive punctures of $\dot \Sigma'$ and we glue along all of them. For complete gluing we write $\T'' = \T \sharp_R \T'$ and $L'' = L \sharp_R L'$.
To get a non-complete gluing from complete gluing, one can add a few trivial CR tuples before gluing in the obvious way. See Figure~\eqref{fig: adding trivial cylinders}. 

\begin{figure} 
\centering
\begin{tikzpicture}[scale = 0.6, every node/.style={scale=0.6}]

\draw (0,4) ellipse (1 and 0.2);
\draw (4,4) ellipse (1 and 0.2);
\draw (12,4) ellipse (1 and 0.2);

\draw (0,1) ellipse (1 and 0.2);
\draw (4,1) ellipse (1 and 0.2);
\draw (8,1) ellipse (1 and 0.2);
\draw (12,1) ellipse (1 and 0.2);
\draw (16,1) ellipse (1 and 0.2);

\draw (0,-1) ellipse (1 and 0.2);
\draw (4,-1) ellipse (1 and 0.2);
\draw (8,-1) ellipse (1 and 0.2);
\draw (12,-1) ellipse (1 and 0.2);
\draw (16,-1) ellipse (1 and 0.2);

\draw (16,-4) ellipse (1 and 0.2);

\draw (4,-4) ellipse (1 and 0.2);
\draw (8,-4) ellipse (1 and 0.2);

\draw (7,1) .. controls (7,2) and (11,3) .. (11,4);
\draw (17,1) .. controls (17,2) and (13,3) .. (13,4);
\draw (1+8,1) .. controls (1+8, 2) and (3+8, 2) .. (3+8,1);
\draw (1+12,1) .. controls (1+12, 2) and (3+12, 2) .. (3+12,1);

\draw (1,-1) .. controls (1, -2) and (3, -2) .. (3,-1);
\draw (5,-1) .. controls (5,-2) and (7,-2) .. (7,-1);
\draw (5+4,-1) .. controls (5+4,-2) and (7+4,-2) .. (7+4,-1);
\draw (5, -4) .. controls (5, -3) and (7,-3) .. (7,-4);

\draw (-1,-1) .. controls (-1,-2) and (3,-3) .. (3,-4);
\draw (13, -1) .. controls (13,-2) and (9,-3) .. (9,-4);

\draw (-1,1) -- (-1,4);
\draw (1,1) -- (1,4);
\draw (3,1) -- (3,4);
\draw (5,1) -- (5,4);

\draw (15, -1) -- (15, -4);
\draw (17, -1) -- (17, -4);

\end{tikzpicture}

\caption[]{Taking disjoint union with trivial cylinders before gluing.}
\label{fig: adding trivial cylinders}
\end{figure}
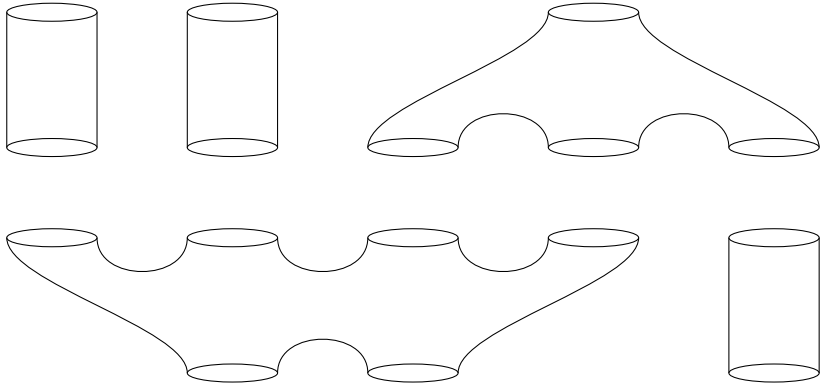


{\em From now on, all gluings are assumed to be complete unless otherwise specified.}

\begin{lemma}[Corollary 7 in \cite{bourgeois2004coherent}, Lemma 9.6 in \cite{hutchings2009gluing} and Lemma A.7 in \cite{hutchings2022s1equivariant}] \label{lemma: gluing}
The gluing map induces an isomorphism $$\det L \otimes \det L' \to  \det (L \sharp_R L'),$$
which is continuous with respect to the homotopies of CR operators $L$ and $L'$, as well as the gluing parameter $R$.
\end{lemma} 
\begin{remark}\label{rmk: different choices of gluing}
    When both $L$ and $L'$ are surjective, the isomorphism $\det L \otimes \det L' \to \det (L \sharp_R L')$ by the following steps:
    \begin{equation}\label{equation: gluing map in basis}
        e_1\wedge \dots \wedge e_n \otimes e_1' \wedge \dots \wedge e_{n'}' \mapsto \tilde e_1 \wedge \dots \wedge \tilde e_n \wedge \tilde e_1' \wedge \dots \wedge \tilde e_{n'}',
    \end{equation}
    where $\tilde e_i$ is obtained by:
    \begin{enumerate}
        \item Translating $e_i$ by $R$ to obtain $e_{i,R}$.
        \item Multiplying $e_{i,R}$ by a cutoff function $\beta_i$ that is 0 near the puncture.
        \item Identifying $\beta_i e_{i,R}$ as a section over $\dot \Sigma''$.
        \item Projecting $\beta_i e_{i,R}$ to $\ker L''$. The image of $\beta_i e_{i,R}$ under the projection is denoted as $\tilde e_i$.
    \end{enumerate}
    Similarly, $\tilde e_j'$ is constructed.
\end{remark} 

Because of Lemma~\ref{lemma: gluing}, we write $\sharp$ instead of $\sharp_R$, meaning gluing with some unspecified gluing parameter.
The gluing map induces an isomorphism $$\det \T \otimes \det \T' \to  \det (\T \sharp \T'),$$
and we denote the image of $v\otimes v'$ under the isomorphism as $v\sharp v'$.

\begin{lemma}[Lemma 9.7 in \cite{hutchings2009gluing}]
The gluing operation for CR tuples is associative up to homotopy,
 as is the induced operation on determinants:
$$(v\sharp v')\sharp v'' = v \sharp (v' \sharp v'')$$
for any $v \in \det \T$, $v' \in \det \T'$, and $v'' \in \det \T''$.
\end{lemma}

\subsection{Construction of coherent orientation systems.}
To construct a coherent orientation system for all CR tuples $\T$ with loops of admissible symmetric matrices from $\{S_1, S_2, \dots\}$, we follow the method outlined in \cite{bourgeois2004coherent}. This involves the following steps: 
\be 
    \item For each loop of admissible symmetric matrices $S$, we perform the following: 
            \be 
            \item Choose a CR tuple $\T^+_S$ with one positive puncture, zero negative punctures, and the associated admissible loop $S$.
            \item Select an orientation $o(\T^+_S)$.
            \item Pick a CR tuple $\T^-_S$ with zero positive puncture and one negative puncture and the associated admissible loop being $S$.
            \item Determine the orientation $o(\T^-_S)$ by gluing $\T^+_S$ and $\T^-_S$ together. Specifically, the equation $o(\T^+_S) \sharp o(\T^-_S) = \ocan(\T^+_S \sharp \T^-_S)$ uniquely determines $o(\T^-_S)$.
        \ee
\ee 

            We refer to $\T_S^\pm$ as the {\em positive (negative) capping CR tuple} and $o(\T^\pm_S)$ as the {\em positive (negative) capping orientation}. 
\be 
    \item[(2)] For an arbitrary CR tuple $\T$, we define $\T^\pm_{\bs S^\pm}$ as the CR tuple obtained by taking the iterated disjoint union of $\T^\pm_{S^\pm_1} \coprod \dots \coprod \T^\pm_{S^\pm_{k^\pm}}$, where $S^\pm_i$ is the admissible loop of symmetric matrices at the $i$-th positive (negative) puncture. Then, the orientation $o(\T^\pm_{\bs S^\pm})$ is determined by the equation:  
        \begin{equation}\label{equation: orientation of bs S}
        o(\T^\pm_{\bs S^\pm}) = \dconst \cdot o(\T^\pm_{S_1^\pm}) \coprod \dots \coprod o(\T^\pm_{S_{k^\pm}^\pm}),
        \end{equation}
        where $\dconst = \dconst(\T^\pm_{S_1^\pm}, \dots, \T^\pm_{S_{k^\pm}^\pm}) \in \{\pm 1\}$ is to be chosen.  
    \item[(3)] Finally, we glue $\T^+_{\bs S^+}$, $\T$, and $\T^-_{\bs S^-}$ together to obtain a CR tuple with no punctures. The orientation $o(\T)$ is determined by the equation: 
    \begin{equation}\label{equation: orientation of T} 
        o(\T_{\bs S^+}^+) \sharp o(\T) \sharp o(\T^-_{\bs S^-}) = \gconst \cdot \ocan (\T_{\bs S^+}^+ \sharp \T \sharp \T^-_{\bs S^-}),
    \end{equation}
    where $\gconst = \gconst(\T^+_{\bs S^+}, \T, \T^-_{\bs S^-}) \in \{\pm 1\}$  is to be chosen. 
\ee 

 The choices of $\dconst$ and $\gconst$ in \cite{bourgeois2004coherent} and in \cite{hutchings2009gluing} are different,
resulting in different properties of the coherent orientations. We list the properties and the choices of $\dconst$ and $\gconst$ of coherent orientations $o_{bm}$ in \cite{bourgeois2004coherent} and $o_{ht}$ in \cite{hutchings2009gluing} below. The signs $\dconst$ and $\gconst$ in Formulas~\eqref{equation: orientation of bs S} and \eqref{equation: orientation of T} come from disjoint union and gluing of multiple (possibly greater than two) CR tuples, respectively. 
They are determined inductively by the signs that come from disjoint union and gluing two CR tuples. 

\begin{theorem}[\cite{bourgeois2004coherent}]
There exists a choice of orientations $o_{bm}$ for all CR tuples such that: \be
    \item The gluing map is orientation-preserving with a sign correction: $$o_{bm}(\T)\sharp o_{bm}(\T') = \gconst_{bm}\cdot o_{bm}( \T \sharp \T'),$$
    where $$\gconst_{bm} = (-1)^{ \sum_{1\leq a < b \leq k_-} |S_a^-|\cdot |S_b^-| }.$$ The sign $(-1)^{ \sum_{1\leq a < b \leq k_-} |S_a^-|\cdot |S_b^-| }$ arises from reversing the ordering of negative ends of $\T$. 
    \item The disjoint union map is orientation-preserving  with a sign correction: 
        $$o_{bm}(\T) \coprod o_{bm} (\T') = \dconst_{bm} \cdot o_{bm}(\T \coprod \T'),$$
        where $$\dconst_{bm} = (-1)^{\op{ind}^-\T \cdot \op{ind}^+ \T'}.$$
    \item When $k_- + k_+ = 0$, $o_{bm}(\T) = \ocan(\T).$ 
    \item An isomorphism between CR tuples preserves orientation. 
\ee 
\end{theorem}
The gluing convention in this paper is different from that of \cite{bourgeois2004coherent}\footnote{The matching ordering between the negative punctures of $\dot \Sigma$ and the positive punctures of $\dot \Sigma'$ is reversed in \cite{bourgeois2004coherent}.}, which is source of the extra sign $\gconst_{bm}$.

\begin{theorem}[\cite{hutchings2009gluing}]\label{thm: ht main}
There exists a choice of orientation $o_{ht}$ for all CR tuples such that: \be
    \item The gluing map is orientation-preserving: 
    $$o_{ht}(\T)\sharp o_{ht}(\T') =  o_{ht}( \T \sharp \T'),$$ 
    which in particular means $\gconst_{ht} = 1$. 
    \item The disjoint union map is orientation-preserving with a sign correction:
        $$o_{ht}(\T) \coprod o_{ht}(\T') = \dconst_{ht} \cdot o_{ht}(\T \coprod \T'),$$
        where $$\dconst_{ht} = (-1)^{\op{ind}^- \T \cdot \op{ind} \T'}.$$
    \item When $k_- + k_+ = 0$, $o_{ht}(\T) = \ocan(\T).$ 
    \item An isomorphism between CR tuples preserves orientation.
\ee 
\end{theorem}

Let $\T'$ be the same CR tuple as $\T$ except that the $i$-th and the $(i+1)$-th positive (or negative) punctures are swapped. Note that $\det \T$ does not depend on the ordering of the punctures of $\T$, so $\det \T = \det \T'$. On the other hand, the coherent orientation depends on the ordering of the punctures. Indeed, 
\begin{corollary} \label{cor: switching ends}
Let $o = o_{bm}$ or $o_{ht}$. Then $$o(\T) = (-1)^{|S^\pm_i|\cdot |S^\pm_{i+1}|} o(\T').$$
\end{corollary}
\begin{proof}
 
The case when $o = o_{bm}$ can be found in \cite{bourgeois2004coherent}, and we do not repeat it here.
We provide a proof for the case when $ o = o_{ht}$ and when $\T$ has two positive punctures and zero negative punctures. The general case can be proved similarly, and we skip it here. 

Let $\T_{S_i^+}^+$ for $i = 1,2$ be the positive capping CR tuple.
Then the two gluings $(\T_{S_1^+}^+ \coprod \T_{S_2^+}^+) \sharp \T $ and $(\T_{S_2^+}^+ \coprod \T_{S_1^+}^+)\sharp \T'$ give the same CR tuple $\T''$ with no punctures. 
Hence, we have 
\begin{align*}
& (-1)^{|S_1^+|\cdot |S_2^+|} \left(  o(\T_{S_1^+}^+) \coprod o(\T_{S_2^+}^+) \right) \sharp o(\T) \\
= & o(\T'') = (-1)^{|S_1^+| \cdot  |S_2^+|} \left(  o(\T_{S_2^+}^+) \coprod o(\T_{S_1^+}^+) \right) \sharp o(\T').
\end{align*}
By Lemma~\ref{lemma: disjoint union}, we have 
$$o(\T_{S_1^+}^+) \coprod o(\T_{S_2^+}^+) = (-1)^{|S^+_1|\cdot |S^+_{2}|} o(\T_{S_2^+}^+) \coprod o(\T_{S_1^+}^+),$$
which implies the statement.

\end{proof}

\section{Coherent orientation of the moduli spaces}

Let $(M, \xi)$ be a contact manifold of dimension $2n-1$ and $\alpha$ be a contact $1$-form such that $\xi = \ker \alpha$. 
Denote by $R_\alpha$ the Reeb vector field of $\alpha$.
We define a Reeb orbit is {\em non-degenerate} if the Poincar\'e return map restricted to $\xi$ along the Reeb orbit does not have $1$ as an eigenvalue.
We make the assumption that all Reeb orbits are non-degenerate.  

For each simple (i.e., not multiply covered) Reeb orbit $\gamma$, we choose a trivialization $\tau$ of the symplectic vector bundle $(\xi, d\alpha |_\xi)$ restricted to $\gamma$. Then the linearized flow of $R_\alpha$ along $\gamma$ gives a path of symplectic matrices, and its Maslov index is called the {\em Conley-Zenhnder index} of $\gamma$, denoted by $\mu_{\op{CZ},\tau} (\gamma)$. We assign to $\gamma$ the $\Z_2$-grading $|\gamma| = \mu_{\op{CZ},\tau} (\gamma) + n-1 \mod 2$, which is independent of the choice of $\tau$.

On each simple Reeb orbit $\gamma$, we choose a fixed point $x_\gamma$, which we call an {\em asymptotic marker}. 

Consider an ($\R$-invariant) $\alpha$-tame almost complex structure $J$ on $W := \R \times M$, where the definition of $\alpha$-tame can be found in \cite{bao2018definition} (Definition 3.1.1).
We recall the definition of moduli spaces of $J$-holomorphic curves.
\begin{definition}[moduli space of $J$-holomorphic curves] \label{def: moduli spaces}
Consider integers $k_+ \geq 1$ and $k_- \geq 0$, 
let $\bs\gamma_\pm = (\gamma_{\pm,1}, \dots, \gamma_{\pm,{k_\pm}})$ denote an ordered tuple of Reeb orbits.
For any $g\in \Z^{\geq 0}$, $A \in H_2(M;\Z)$, we define the moduli space of $J$-holomorphic curves  $\parametrizedModulispace{g}{\gamma_+}{\gamma_-}{A}$  consisting of equivalence classes $[\Sigma, j, \bs p, \bs r, u]$ of tuples satisfying the following conditions: 
\be 
    \item $(\Sigma, j, \bs p, \bs r)$ is a {\em connected}, decorated Riemann surface of genus $g$ with $k_+$ positive and $k_-$ negative marked points. 
\ee 
Let $\phi_i^\pm: D\subset \C \to \mathcal U_{p_i^\pm}\subset \Sigma$ be a biholomorphic map from the unit disc $D$ to a neighborhood $\mathcal U_{p_i^\pm}$ of $p^\pm_i$. We require $\phi_i^\pm(o) = p_i^\pm$ and $d\phi_i^\pm(\frac{\partial}{\partial x}) = r^\pm_i,$ where $\frac{\partial}{\partial x} \in T_o D.$ 
Additionally, let $h_i^\pm:\R^{\geq 0 (\leq 0)} \times S^1 \to \dot{\mathcal U}_{p_i^\pm}$ be the biholomorphic map defined by $h_i^\pm (s,t) =(\phi^\pm_i)^{-1}(e^{\mp s \mp \sqrt{-1} t})$.

\be 
    \item[(2)] $u:\Sigma - \bs p \to W$ is a proper map satisfying: 
        \be
            \item $\overline \partial_J u = \frac{1}{2}(du + J \circ du \circ j) = 0$.
            \item $\lim_{s \to \pm \infty} u\circ (\phi^\pm_i)^{-1}(e^{\mp s \mp \sqrt{-1} t}) = \gamma_{\pm,i} ( T  t),$ where $T > 0$ is the period of $\gamma_{\pm,i}$  and $\gamma_{\pm,i}$ is parametrized such that $\gamma_{\pm,i} (0) = r_i^\pm$.

        \ee
    \item[(3)] The homology class obtained by ``capping off" the punctures of $u$ is $A$. Refer to Section~9.1.2 in \cite{bao2015semi} or \cite{bourgeois2004coherent} for the ``capping off" construction.
    
    \item[(4)] $(\Sigma, j, \bs p, \bs r, u)$ is equivalent to $(\Sigma', j', \bs p', \bs r', u')$ if there exists a biholomorphic map $f: (\Sigma, j) \to  (\Sigma', j')$ such that $f$ also maps $(\bs p, \bs r)$ to $(\bs p', \bs r')$ and $u' =  u \circ f.$ 
\ee  
We define the quotient space as  $$\unparametrizedModulispace{g}{\gamma_+}{\gamma_-}{A} = \parametrizedModulispace{g}{\gamma_+}{\gamma_-}{A}/\R,$$
where $\R$ acts by composing the map $u$ with the translation in the $\R$-direction in $W$.

\end{definition}

It is convenient to have the following lemma.
\begin{lemma}\label{lemma: stabilize in the domain}
Let $U$ and $W$ be Banach spaces, and let $\phi: U\to W$ be a linear Fredholm operator. Let $V$ be a finite-dimensional vector space and $\psi: V\to W$ be a linear map. If the map $\phi \oplus \psi: U \oplus V \to W$ defined as $\phi\oplus \psi (u, v) = \phi(u) + \psi(v)$ is surjective,
then there exists an isomorphism 
$$\det \phi \simeq \det (\phi \oplus \psi) \otimes \wedge^{\op{top}} V^*,$$
which is natural up to a positive constant.
\end{lemma}
\begin{proof} See Exercise A.23 in \cite{mcduff2012j} for the case when $\psi$ is  injective.
Let $I = \phi(U) \cap \psi(V) \subset W$.
Suppose $\phi^{-1}(I) = H \oplus \ker \phi$ and 
let $\{u_1, \dots, u_k\}$ be a basis of $H$,
and $\{u_{k+1}, \dots, u_{k+\ell}\}$ be a basis of $\ker \phi$.
Then $\{\phi(u_1),\dots, \phi(u_k)\}$ forms a basis of $I$.
Suppose $V = G \oplus \psi^{-1}(I)$ and $\psi^{-1}(I) = F \oplus \ker \psi$.
Let $$\{v_1, \dots v_m, v_{m+1}, \dots , v_{m+k}, v_{m + k +1}, \dots, v_{m+k+n}\}$$ be a basis of $V$ such that
\be
    \item $\{v_1, \dots, v_m\}$ is a basis of $G$,
    \item $\{v_{m+1},\dots, v_{m +k}\}$ is a basis of $F$ and $\psi(v_{m+i}) = \phi(u_{i})$ for $i = 1, \dots, k$, and
    \item $\{v_{m+k+1}, \dots, v_{m+k+n}\}$ is a basis of $\ker \psi$. 
\ee 
The isomorphism is given by: 
\begin{align*}
    & u_{k+1} \wedge \dots \wedge u_{k + \ell} \otimes \psi(v_m)^* \wedge \dots \wedge \psi(v_1)^*  \\
    \mapsto & (u_{k+1}, 0) \wedge \dots \wedge (u_{k+\ell}, 0) \wedge (u_1, -v_{m+1}) \wedge \dots \wedge (u_k, -v_{m+k}) \\
    & \wedge (0, v_{m+k+1}) \wedge \dots \wedge (0, v_{m+k+n}) \otimes v_{m+k+n}^* \wedge \dots \wedge v_1^*.
\end{align*}
\end{proof}

For any $$[\Sigma, j, \bs p, \bs r, u] \in \unparametrizedModulispace{g}{\gamma_+}{\gamma_-}{A},$$
consider the complex vector bundle $E:= u^*TW$ over $\dot \Sigma$ where the complex structure is given by $u^*J.$
Around each Reeb orbit $\gamma_{\pm,i} \in \bs \gamma_\pm$, a trivialization $\tau$ of $(\xi, d\alpha|_\xi, J)$ restricted to $\gamma_{\pm,i}$ (induced from the trivialization over the underlying simple Reeb orbit) extends to a trivialization of $(E, J)|_{\dot{\mathcal U}_{p_i^\pm}}$ as in Formula~\eqref{formula: trivialization of E}. This gives a CR tuple denoted by $\T_u$. A different representative of $[\Sigma, j, \bs p, \bs r, u]$ gives an isomorphic CR tuple.
Let $$D_u: C^\infty  (\dot \Sigma,  E)\to C^\infty (\dot \Sigma, \wedge^{0,1}\dot \Sigma \otimes_\C E)$$ be the linearized $\overline \partial_J$ operator. It extends  to a Fredholm operator as in Formula~\eqref{formula: Fredholm operator}. 
Hence, we get a CR operator denoted by $D_u \in \mathfrak D(\T_u)$.

Fix any $[u] \in \parametrizedModulispace{g}{\gamma_+}{\gamma_-}{A}$. If $6g-6 + 2(k_+ + k_-) > 0$, we have the full linearized $\overline \partial_J$ operator at $[u]$
\[
\mathcal D: \textit{Teich} \oplus  C^\infty  (\dot \Sigma,  E)\to  C^\infty (\dot \Sigma, \wedge^{0,1}\dot \Sigma \otimes_\C E),
\]
where $\textit{Teich}$ is the tangent space of complex structures of $(\Sigma, \bs p)$ at $j$, 
which has dimension equal to  $6g-6 + 2(k_+ + k_-)$ and $\mathcal D|_{ C^\infty  (\dot \Sigma,  E)} = D_u$.
Suppose that $[u]$ is transversely cut out, i.e., $\op{coker} \mathcal D = \{0\}$.
Then 
$$\wedge^{\op{top}}\left(T_{[u]}\parametrizedModulispace{g}{\gamma_+}{\gamma_-}{A}\right) \simeq \wedge^{\op{top}} \ker \mathcal D \simeq \det D_u \otimes \wedge^{\op{top}} \textit{Teich}^*,$$ 
where the last isomorphism is given by Lemma~\ref{lemma: stabilize in the domain}. 

If $6g-6 + 2(k_+ + k_-) < 0$ and $[u]$ is transversely cut out, i.e., $\op{coker}D_u = \{0\}$, then: 
$$\ker D_u  =  T_{[u]}\parametrizedModulispace{g}{\gamma_+}{\gamma_-}{A}\oplus \textit{Aut},$$ 
where 
$\textit{Aut}$ is the group of the biholomorphism of $(\dot\Sigma, j)$ with dimension $-(6g-6 + 2 (k_+ + k_-))$.

In either case, the virtual dimension of the moduli space is given by:
\begin{align*} 
&  \op{virdim} \parametrizedModulispace{g}{\gamma_+}{\gamma_-}{A} \\
= & \dim_\R \ker D_u + 6g-6 + 2(k_+ + k_-) \\
=  & \sum_{i=1}^{k_+} \mu_{\op{CZ},\tau}(\gamma_{+,i}) - \sum_{i=1}^{k_-} \mu_{\op{CZ},\tau}(\gamma_{-,i})  + 2c_1(E;\tau) + (n-3) (2-2g - k_- - k_+),
\end{align*}
where $c_1(E;\tau)$ is the relative first Chern number. 
Moreover, we have the canonical isomorphism: 
\begin{equation}\label{formula: isomorphism from kerDu to tangent space}
\det \T_u \simeq \wedge^{\op{top}} (T_{[u]}\parametrizedModulispace{g}{\gamma_+}{\gamma_-}{A})
\end{equation}
as the portions $\textit{Teich}$ and $\textit{Aut}$ are complex and hence canonically oriented.
Let $o = o_{ht}$ or $o_{bm}$.
Note that:
\begin{equation}
    T_{[u]}\parametrizedModulispace{g}{\gamma_+}{\gamma_-}{A} \simeq \R \langle \partial_s ([u]) \rangle \oplus T_{[u]} \unparametrizedModulispace{g}{\gamma_+}{\gamma_-}{A},
\end{equation} 
  where $\partial_s$ is the vector field on $\parametrizedModulispace{g}{\gamma_+}{\gamma_-}{A}$ that is generated by the $\R$-translation, and $\partial_s ([u])$ is $\partial_s$ evaluated at $[u]$.  We define the orientation $\tilde o(u)$ of $T_{[ u ]} \unparametrizedModulispace{g}{\gamma_+}{\gamma_-}{A}$ by the equation:
\begin{equation} \label{formula: definition of orientation of moduli space}
o(\T_u) \simeq \partial_s([u]) \wedge \tilde o(u),
\end{equation}
where the isomorphism is given by Formula~\eqref{formula: isomorphism from kerDu to tangent space}. 
In the case when $$\op{virdim} \parametrizedModulispace{g}{\gamma_+}{\gamma_-}{A} = 1,$$ $\tilde o(u) \in \{\pm 1\}$.

For any $[ u ] \in \widetilde{\mathcal M}_1 = \parametrizedModulispace{g}{\gamma_+}{\gamma_-}{A}$ and $[ v ] \in \widetilde{\mathcal M}_2 = \parametrizedModulispace{g'}{\gamma_+'}{\gamma_-'}{A'}$ that are transversely cut out and the first $\tau$ Reeb orbits of $\bs \gamma_-$ match the last $\tau$ Reeb orbits in $\bs\gamma_+'$,
we can glue $u$ and $v$ with some fixed large gluing parameter along the $\tau$ punctures to $[ w ] \in \widetilde{\mathcal M}_3 = \parametrizedModulispace{g''}{\gamma_+''}{\gamma_-''}{A''}$,
where $g'' = g + g' + (\tau - 1)$, $$\bs \gamma_+'' = (\gamma_{+,1}',\dots, \gamma'_{+,k'_+-r}, \gamma_{+,1}, \dots, \gamma_{+,k_+}),$$
$$\bs \gamma_-'' = (\gamma_{-,1}', \dots, \gamma_{-,k_-}',\gamma_{-,k_--\tau + 1}, \dots, \gamma_{-, k_-}),$$
and $A'' = A + A'$. 
The construction of the gluing map on the moduli spaces is standard (see for example Section 10 in \cite{mcduff2012j} or Section 6 in \cite{bao2015semi}). The gluing of the CR tuples in Section~\ref{section: gluing} is the linearized version of this, and indeed, we have the commutative diagram:
\begin{equation} \label{formula: commutative diagram}
\begin{tikzcd}
\det \T_u \otimes \det \T_v   \arrow[r,"\sharp_\tau"] \arrow[d, "\simeq"] & \det \T_w \arrow[d, "\simeq"]\\
 \wedge^{\op{top}} ( T_{[u]}\widetilde{\mathcal M}_1) \otimes \wedge^{\op{top}} (T_{[v]} \widetilde{\mathcal M}_2)    \arrow[r] &  \wedge^{\op{top}}(T_{[w]}\widetilde{\mathcal M}_3),
\end{tikzcd}
\end{equation}
where the lower horizontal map is induced by gluing.
 Let $\mathcal M_i = \widetilde{\mathcal M}_i/\R$ for $i = 1,2,3$.  
The curve $[w]$ is transversely cut out and can be viewed as in a codimension one boundary component of  (a retract of) $\mathcal M_3$. Thus, we have the boundary orientation $\tilde o^b(w) \in \{\pm 1\}$ defined by: 
\begin{equation} \label{formula: boundary orientation}
     \tilde o(w) = n \wedge \tilde o^b (w),
\end{equation}
where $n$ is the outward-pointing normal (pointing in the gluing parameter increasing direction). 

Denote $o(\T_u) \simeq \partial_s([u]) \wedge \tilde o(u)$, $o(\T_v) \simeq \partial_s([v]) \wedge \tilde o(v)$, and $o(\T_{w}) \simeq \partial_s ([w]) \wedge \tilde o(w)$.  Then we have 
$$o(\T_u) \sharp_\tau o(\T_v) = \gconst \cdot \dconst \cdot o(\T_u \sharp_\tau \T_v) = \gconst \cdot \dconst \cdot o(\T_{w }),$$
where the formula for partial gluing follows from taking disjoint unions with trivial tuples before gluing, and more precisely
\be
    \item if $o = o_{ht}$, then 
    \be 
        \item $\gconst = 1$, and
        \item $\dconst = (-1)^{(\sum_{i=1}^{k_--\tau} |S'_{+,i}|) \cdot \op{ind}\T_u}$;
    \ee 
    \item if $o = o_{bm}$, then 
    \be 
        \item $\gconst = \epsilon_{bm}^\sharp$, and
        \item $\dconst = (-1)^C$ with 
        \begin{align*}
        C = & \sum_{1\leq i < j \leq \tau} |S'_{+,i}|\cdot |S'_{+,j}| + (\sum_{i=1}^\tau |S'_{+,i}|)(\sum_{i=1}^{k_+}|S_{+,i}|) \\
         & + \sum_{\tau + 1 \leq i < j \leq k_-} |S_{-,i}|\cdot |S_{-,j}| + (\sum_{i=1}^{k'_-} |S'_{-,i}| )(\sum_{i=\tau+1}^{k_-}|S_{-,i}|).
        \end{align*}
    \ee 
\ee
Tracking the images of $o(\T_u)\otimes o(\T_v) \in \det \T_u \otimes \det \T_v$ in the commutative diagram~\eqref{formula: commutative diagram}, we have 
\begin{align*}
\gconst \cdot \dconst \cdot \partial_s([w]) \wedge n \wedge \tilde o^b (w)& = \widehat{\partial_s([u])} \wedge \widehat{\tilde o(u)} \wedge \widehat{\partial_s([v])} \wedge \widehat{\tilde o(v)} \\
& =  (-1)^{\op{ind}\T_u-1}  \widehat{\partial_s([u])} \wedge \widehat{\partial_s([v])} \wedge \widehat{\tilde o(u)} \wedge  \widehat{\tilde o(v)} \\
& = (-1)^{\op{ind}\T_u} \partial_s([w]) \wedge n \wedge \widehat{\tilde o(u)} \wedge  \widehat{\tilde o(v)},
\end{align*}
where: 
\be 
    \item $\widehat{\partial_s([u])}\in \ker D_w$ is defined as follows (see, for example, Section 9.12 in \cite{bao2018definition} for details in the case when the domains are cylinders): 
    \be 
        \item Translate $\eta = \partial_s([u])$ in the $\R$-direction by the gluing parameter $R$, yielding $\eta_R$.
        \item Choose a cutoff function $0\leq \beta \leq 1$ that is equal to $0$ in a small neighborhood of punctures and is equal to $1$ outside a small neighborhood of punctures.
        \item Multiply $\eta_R$ with $\beta$ to damp it out near punctures. 
        \item View $\beta \eta_R$ as an element in the domain of $D_w$, the linearized $\overline \partial$ operator at $[w]$.
        \item Project $\beta \eta_R$ to $\ker D_w$ with respect to the $L^2$ inner product. 
    \ee 
    \item $\widehat{\partial_s([v])}$ is defined similarly, except that in step (1), it is translated by $-R$.
    \item Supposing $\tilde o(u) = v_1 \wedge \dots \wedge v_k$ with $k = \op{ind}\T_u - 1 \geq 1$, $v_i \in T_{[u]} \mathcal M_1$ for $i = 1, \dots, k$, the term $\widehat{\tilde o(u)}$ is defined as $\widehat v_1 \wedge \dots \wedge \widehat v_k$, where $\widehat v_i$ is defined in the same way as in (1). If $\op{ind} \T_u = 1$, we define $\widehat{\tilde o(u)} = {\tilde o(u)}\in \{\pm 1\}$.
    \item  The last equality follows from the fact that, up to multiplication by a positive number, $\partial_s([w])$ is approximately $\widehat{\partial_s([u])} + \widehat{\partial_s([v])}$ and $n$ is approximately $\widehat{\partial_s([u])} - \widehat{\partial_s([v])}$, and hence $\partial_s([w]) \wedge n$ is approximately $- \widehat{\partial_s([u])} \wedge \widehat{\partial_s([v])}$.
\ee 
In summary, we have the following lemma.
\begin{lemma} \label{lemma: boundary orientation}
The gluing map $\mathcal M_1 \times \mathcal M_2 \to \partial \mathcal M_3$
changes the orientation by the sign $\gconst \dconst \cdot (-1)^{\op{virdim}\mathcal M_1 + 1} = \gconst \dconst \cdot (-1)^{\op{ind} \T_u}.$
\end{lemma}
In particular, if we use the orientation $o_{ht}$, in the proof of $\Ha \cdot \Ha = 0$ (See Section~\ref{section: sft}), where $\op{virdim} \mathcal M_1= 0$, we have the gluing map reverses the boundary orientation.

\section{Signs in symplectic field theory} \label{section: sft}



For each Reeb orbit $\gamma$, we assign two formal variables $p_\gamma$ and $q_\gamma$, graded over $\Z_2$ by $|\gamma|$.
For any $A\in H_2(M;\Z)$, we represent it multiplicatively as $e^A$ and grade it by $0$.
Let $\hbar$ be a formal variable graded by $0$ to keep track of the genus $g$.
Consider  the Weyl super-algebra  $$\mathfrak W = \Q[\{q_\gamma\}_\gamma, \hbar,  \{e^A\}_{A \in H_2(M;\Z)}]\llbracket  \{p_\gamma\}_\gamma \rrbracket,$$ which is the space of all formal power series in $\{p_\gamma\}_\gamma$ over the polynomial ring $$\Q[\{q_\gamma\}_\gamma,  \{e^A\}_{A \in H_2(M;\Z)}, \hbar].$$
We require that all formal variables are graded commutative except 
\begin{equation} \label{formula: graded commutative} 
[p_\gamma, q_{\gamma} ] = p_\gamma q_{\gamma} - (-1)^{|\gamma|} q_\gamma p_\gamma = \frac{\hbar}{m(\gamma)}, 
\end{equation}
where $m(\gamma)$ is the multiplicity of $\gamma$ over the underlying simple Reeb orbit.
In \cite{eliashberg2000introduction},
a potential function $\Ha$ is constructed by counting $J$-holomorphic curves. 
We recall and modify the definition as follows: 
\begin{equation} \label{formula: H_g}
\Ha = \sum_{g \geq 0} \sum_{[\bs \gamma_+], [\bs \gamma_-]} \sum_{A \in H_2(M)}|\unparametrizedModulispace{g}{\gamma_+}{\gamma_-}{A} |  \ q_{\bs \gamma_-}  p_{\bs\gamma_+^\dagger}   e^A \hbar^{g-1} ,
\end{equation}
where 
\be 
\item the second summation is over pairs of unordered tuples  $[\bs\gamma_\pm]$ of Reeb orbits,
\item $\bs\gamma_\pm = (\gamma_{\pm, i_1}, \dots , \gamma_{\pm, i_{k_\pm}})$ is an ordered tuple of Reeb orbits representing the equivalence class $[\bs\gamma_\pm]$,
\item $\bs \gamma_+^\dagger = (\gamma_{+,i_{k_+}}, \dots, \gamma_{+, i_1})$ is the ordered tuple obtained from $\bs\gamma_+$ by reversing the ordering,
\item $p_{\bs\gamma_+^\dagger}$ (resp. $q_{\bs\gamma_-}$) is the monomial of $p_\gamma$ (resp. $q_\gamma$) that is associated to the ordered tuple $\bs \gamma_+^\dagger$ (resp. $\bs\gamma_-$), 
\item $m(\bs\gamma_-) = m(\gamma_{-,i_1})\dots m(\gamma_{-,i_{k_-}})$ , and 
\item $|\unparametrizedModulispace{g}{\gamma_+}{\gamma_-}{A} |$ is the signed count of elements in the moduli space based on the coherent orientation $o_{ht}$ and is set to be $0$ when the virtual dimension is not $0$.
\ee 
Note that Formula~\eqref{formula: H_g} does not depend on the choice of representatives $\bs \gamma_\pm$ of $[\bs\gamma_\pm]$ by Corollary~\ref{cor: switching ends}.

\begin{remark}
The sign correction of $\Ha$ as mentioned in the abstract is the usage of $p_{\bs\gamma_+^\dagger}$ over $p_{\bs\gamma_+}$. 
\end{remark}
To ensure that $\Ha$ is well-defined, one needs to perturb the moduli space because the multiply covered curves are not transversely cut out in general. 
Several versions of perturbation theories fit or can be generalized to fit this setting, including but not limited to  \cite{pardon2019contact, bao2015semi, ishikawa2018construction, hofer2007general}. Transversality is far beyond the scope of this paper.
\begin{theorem}\label{theorem: main}
Suppose that the moduli spaces are transversely cut out after some perturbation. We have the product 
$$\Ha \cdot \Ha = 0.$$ 
\end{theorem}
We define the differential $D: \mathfrak W \to \mathfrak W$ by $Df = [\Ha, f]$ for all $f \in \mathfrak W$, and 
the homology algebra $H_*(\mathfrak W, D)$ to be the homology of $(\mathfrak W, D)$.
Before proving the theorem, we first revisit  the example in Figure 4 of \cite{eliashberg2000introduction} with some modifications. 
\begin{example}
Suppose $$\Ha = a q_1 q_2 p_4 \hbar^{-1} + bp_3 p_2 p_1 \hbar^{-1},$$
where: 
\begin{itemize}
    \item $a = |\mathcal M^0(\gamma_4;\gamma_1\gamma_2) |$ with $|\gamma_4| + |\gamma_1| + |\gamma_2| = 1 \mod 2$.
    \item $b =  |\mathcal M^0(\gamma_1\gamma_2\gamma_3; \emptyset)|$ with $|\gamma_1| + |\gamma_2| + |\gamma_3| = 1 \mod 2$. 
    \item We assume $m(\gamma_i) = 1$, for all $i\in \{1,2,3,4\}$. 
    \item We write $q_i$ and $p_i$ for $q_{\gamma_i}$ and $p_{\gamma_i}$ respectively.
    \item  We drop the variable $e^A$ for $A \in H_2(M;Z)$. 
\end{itemize}
An explicit calculation yields 
\begin{align*}
    \Ha \cdot \Ha = & (-1)^{d_2 + d_2 d_3 + d_2 d_4 + d_3 d_4}ab q_2 p_4 p_3 p_2  \hbar^{-1} \\
                    & + (-1)^{d_1 + d_1 d_2 + d_1 d_4 + d_3 d_4}ab q_1 p_4 p_3 p_1  \hbar^{-1}   \\
                   & + (-1)^{d_3 d_4}ab p_4 p_3 ,
\end{align*}
where $d_i = |\gamma_i|$, and two monomials that are multiples of $q_1 q_2 p_4 p_3 p_2 p_1$ cancel out. 
The three terms that appear in $\Ha \cdot \Ha$ correspond to the three gluings of the moduli spaces  $\mathcal M_{I} = \mathcal M^0(\gamma_4;\gamma_1\gamma_2)$ and $\mathcal M_{\textit{II}} = \mathcal M^0(\gamma_1\gamma_2\gamma_3; \emptyset)$ with signs (see Figure~\ref{fig: gluing}).

We verify the signs of the first term, leaving the other two terms to the reader.   
Consider the moduli space $\mathcal M_{\textit{III}} = \mathcal M^0(\gamma_2\gamma_3\gamma_4;\gamma_2)$. 
We check that the number of elements in $\partial_{\textit{I,II}}\mathcal M_{\textit{III}}$, the part of the boundary of $\mathcal M_{\textit{III}}$ that corresponds   to the gluing of the moduli spaces $\mathcal M_{\textit{I}}$ and $\mathcal M_{\textit{II}}$ along $\gamma_1$, equals $(-1)^{d_2 + d_2 d_3 + d_2 d_4 + d_3 d_4 + 1}ab$, which is $(-1)$ times the coefficient of $q_2p_4p_3p_2 \hbar^{-1}$. 
\begin{align*}
 |\partial_{\textit{I,II}} \mathcal M_{\textit{III}}|  & = \sum_{[w] \in \partial_{\textit{I,II}} \mathcal M_{\textit{III}}} \tilde o^b(w) \\
 & = \sum_{[v]\in \mathcal M(231;\emptyset)}  \sum_{[u] \in \mathcal M(4;12)} (-1)^{d_2 + d_3 + 1} \tilde o(u) \tilde o(v) \\
  & = \sum_{[v]\in \mathcal M(123;\emptyset)} \sum_{[u] \in \mathcal M(4;12)}  (-1)^{d_2 + d_3 + 1 + d_1(d_2 + d_3)} \tilde o(u) \tilde o(v) \\
 & = (-1)^{ d_2 + d_2 d_3 + d_2 d_4 + d_3 d_4 + 1}   \left( \sum_{[v]\in \mathcal M(123;\emptyset)} \tilde o(v) \right) 
 \left(\sum_{[u] \in \mathcal M(4;12)} \tilde o(u) \right)\\
 & = (-1)^{ d_2 + d_2 d_3 + d_2 d_4 + d_3 d_4 + 1} ab,
\end{align*}
where we omit $\gamma'$s in the notation for the moduli spaces;
the second equality follows from Lemma~\ref{lemma: boundary orientation} with $\gconst = 1, \dconst = (-1)^{(d_2 + d_3)\cdot 1},$ and $(-1)^{\op{virdim} \mathcal M_I} = -1$, noting that the last positive end of $\mathcal M(231;\emptyset)$  matches  the first negative end of $\mathcal M(4;12)$;
the fourth equality follows from the fact that $(-1)^{d_2 + d_3 + 1 + d_1(d_2 + d_3)} = (-1)^{ d_2 + d_2 d_3 + d_2 d_4 + d_3 d_4 + 1}$, since $d_4 + d_1 + d_2 = 1 \mod 2$ and $d_1 + d_2 + d_3 = 1 \mod 2$. 
We leave the computation of the other two terms for the reader. 

\begin{figure} 
\centering
\begin{tikzpicture}[scale = 0.5, every node/.style={scale=0.6}]

\draw (0,1) ellipse (1 and 0.2);
\draw (2,4) ellipse (1 and 0.2);
\draw (4,1) ellipse (1 and 0.2);

\draw (0,-1) ellipse (1 and 0.2);
\draw (4,-1) ellipse (1 and 0.2);
\draw (8,-1) ellipse (1 and 0.2);

\draw (-1,1) .. controls (-1,2) and (1,3) .. (1,4);
\draw (5,1) .. controls (5,2) and (3,3) .. (3,4);
\draw (1,1) .. controls (1, 2) and (3, 2) .. (3,1);

\draw (-1,-1) .. controls (-1,-5) and (9,-5) .. (9,-1);
\draw (1,-1) .. controls (1, -2) and (3, -2) .. (3,-1);
\draw (5,-1) .. controls (5,-2) and (7,-2) .. (7,-1);

\node[] at (2, 4.5) {$\gamma_4$};
\node[] at (0, 0) {$\gamma_1$};
\node[] at (4, 0) {$\gamma_2$};
\node[] at (8, 0) {$\gamma_3$};

\end{tikzpicture}

\caption[]{There are three gluings that correspond to three terms in $\Ha \cdot \Ha$: gluing along $\gamma_1$, gluing along $\gamma_2$, and simultaneously gluing along $\gamma_1$ and $\gamma_2$.}
\label{fig: gluing}
\end{figure}
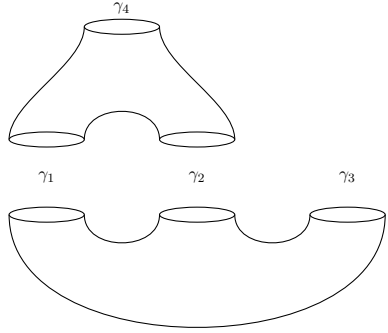
\end{example}

\begin{proof}[Proof of Theorem~\ref{theorem: main}]
Note that 
\[ 
\Ha = \sum_{g \geq 0} \sum_{[\bs \gamma_+], [\bs \gamma_-]} \sum_{A \in H_2(M)} \sum_{[u]\in \unparametrizedModulispace{g}{\gamma_+}{\gamma_-}{A}} \tilde o(u)  \ q_{\bs \gamma_-}  p_{\bs\gamma_+^\dagger}   e^A \hbar^{g-1} ,
\]
and hence, 
\begin{equation}\label{eqn: H H}
    \Ha \cdot \Ha = \sum  \tilde o(u') \tilde o(u) q_{\bs \gamma'_-}  p_{\bs{\gamma'}_+^\dagger} q_{\bs \gamma_-}  p_{\bs\gamma_+^\dagger}    e^{A'+A} \hbar^{g'+g-2},
\end{equation} 
where $\tilde o(u), \tilde o(u') \in \{\pm 1\}$,
and the summation is over $g, g' \geq 0$, unordered tuples of good Reeb orbits $[\bs\gamma_+], [\bs\gamma_-], [\bs\gamma'_+], [\bs\gamma'_-]$, homology classes $A, A' \in H_2(M)$,  and  $[ u ]  \in \unparametrizedModulispace{g}{\gamma_+}{\gamma_-}{A}$ and $ [ u' ] \in \unparametrizedModulispace{g'}{\gamma_+'}{\gamma_-'}{A'}.$ 

It is convenient to choose an ordering for all Reeb orbits. Then, Equation~\eqref{eqn: H H} can be simplified by moving $q_{\bs \gamma'_-}$ to the left of $p_{\bs\gamma_+^\dagger}$ using Equation~\eqref{formula: graded commutative} and subsequently sorting the $q$ terms and the $p$ terms based on the ordering of Reeb orbits.

For any sorted tuples $\bs \gamma''_-$ and ${\bs\gamma''}_+$, $A'' \in H_2(M)$, and $g'' \geq 0$, the coefficient in front of $q_{\bs \gamma''_-}  p_{{\bs\gamma''}_+^\dagger} e^{A''} \hbar^{g''-1}$ within $\Ha \cdot \Ha$ is given by:
\begin{equation}\label{formula: a term of H H}
    \sum \epsilon(u',u) \tilde o(u') \tilde o(u)
\end{equation}
where:
\begin{enumerate}
    \item The sum is taken over all triples $([u], [u'], \vartheta)$ satisfying:
        \begin{enumerate}
            \item $[u] \in \unparametrizedModulispace{g}{\gamma_+}{\gamma_-}{A}$ for some $g, \bs \gamma_+, \bs \gamma_-, A$.
            \item $[u'] \in \unparametrizedModulispace{g'}{\gamma'_+}{\gamma'_-}{A'}$ for some $g', \bs \gamma'_+, \bs \gamma'_-, A'$.
            \item $\vartheta$ is a bijective map from a subset $G_- \subset \{1,\dots, k_-\}$ to a subset $G_+' \subset \{1, \dots, k_+'\}$ such that the $i$-th element of $\bs \gamma_-$ is equal to the $\vartheta(i)$-th element of $\bs \gamma_+'$, for all $i \in G_-$, where $k_-$ is the length of the tuple $\bs \gamma_-$ and $k_+'$ is the length of the tuple $\bs \gamma'_+$.
            \item If we glue $u$ and $u'$ along $\vartheta$ and reorder the punctures, if necessary, we obtain a curve $w = u \sharp_\vartheta u'$ that satisfies $[w] \in \unparametrizedModulispace{g''}{\gamma_+''}{\gamma_-''}{A''}$.
        \end{enumerate}
    \item The number $\epsilon(u',u) \in \Q$ arises from the algebraic operation of moving $q_{\bs \gamma'_-}$ to the left of $p_{\bs\gamma_+^\dagger}$ using Equation~\eqref{formula: graded commutative} and subsequently sorting the $q$ terms and the $p$ terms.
\end{enumerate}
We state the following claim:
\begin{claim}
    The equation $\sum \epsilon(u',u) \tilde o(u') \tilde o(u) = - \sum \tilde o^b(w)$ holds, where the left-hand side represents the term in Formula~\eqref{formula: a term of H H}, and the summation on the right-hand side is taken over all $[w] \in \partial \unparametrizedModulispace{g''}{\gamma''_+}{\gamma''_-}{A''}$.
\end{claim}
Assuming this claim, we can establish the theorem, as $\sum \tilde o^b(w) = 0$.
\end{proof} 
\begin{proof}[Proof of the Claim:]
This is a straightforward calculation. To initiate the proof, we introduce some notations. We define the index set $\{1,\dots, k_-\}$ as $G_- \sqcup N_-$, where $G_-$ represents the set of punctures involved in gluing, and $N_-$ denotes the set of non-gluing punctures. Similarly, we denote $\{1,\dots, k_+'\}$ as $G_+' \sqcup N_+'$.
Consider the term $q_{\bs \gamma'_-}  p_{\bs{\gamma'}_+^\dagger} q_{\bs \gamma_-} p_{\bs\gamma_+^\dagger}$. We move the $q_{\bs \gamma_-}$ term across the $p_{\bs{\gamma'}_+^\dagger}$ in steps. Through this process, the monomial becomes a polynomial by Formula~\eqref{formula: graded commutative}, and we only focus on the term that is prescribed by $\vartheta$, meaning when a $p$ is next to a $q$ term and they are matched in $\theta$, we cancel them; otherwise, they are graded commutative. 

We now analyze the sign $\epsilon(u',u)$ in several steps:
\begin{enumerate}[label = (\arabic*)]
    \item \label{step S1} We move the $q$ terms labeled by $N_-$ to the end of $q_{\bs \gamma_-}$, resulting in a sign denoted by $\epsilon_1$.
    \item \label{step S2} Similarly, we move the $p$ terms labeled by $N'_+$ to the end of $p_{\bs{\gamma'}_+^\dagger}$, obtaining a sign $\epsilon_2$.
    \item \label{step S3} Next, we sort the $p$ terms labeled by $G'_+$ in reverse order according to $\vartheta$, yielding a sign $\epsilon_3$.
    \item \label{step S4} We further move the $p$ terms labeled by $N'_+$ to the end of the monomial, resulting in a sign $\epsilon_4 = (-1)^{\sum_{i\in N'_+} |\gamma'_{-,i}|}$.
    \item \label{step S5} We cancel out the $p$ terms labeled by $G_+'$ with the corresponding $q$ terms labeled by $G_-$ and get a monomial with $q$ terms before the $p$ terms together with a positive factor $\epsilon_5 \in \Q^{>0}$ due to the multiplicity of the Reeb orbits.
    \item \label{step S6} Finally, we sort the $p$ terms and $q$ terms respectively, and get a sign $\epsilon_6$.
\end{enumerate}
In summary, we have $\epsilon(u',u) = \epsilon_1 \epsilon_2 \epsilon_3 \epsilon_4 \epsilon_5 \epsilon_6$. Next, we calculate the sign that arises from gluing $u$ and $u'$ according to $\vartheta$ using Lemma~\ref{lemma: boundary orientation}:
\begin{enumerate}[label = (\arabic*)]
    \item We reorder the negative ends of $u$ by moving the punctures labeled by $N_-$ to the end, while still denoting the curve as $u$. This results in a sign denoted by $\varepsilon_1$.
    \item Similarly, we reorder the positive ends of $u'$ by moving the punctures labeled by $N_+'$ to the front, while still denoting the curve as $u'$. This results in a sign $\varepsilon_2$.
    \item We sort the positive punctures of $u'$ labeled by $G'_+$ according to $\vartheta$, while still denoting the curve as $u'$. This results in a sign $\varepsilon_3$.
    \item We glue the two curves $u$ and $u'$ using Lemma~\ref{lemma: boundary orientation}, and obtain a sign $\varepsilon_4$.
    \item Finally, we sort the positive punctures and negative punctures of the glued curve respectively, getting a sign denoted by $\varepsilon_6$.
\end{enumerate}
It is evident that $\epsilon_1 = \varepsilon_1$, $\epsilon_2 = \varepsilon_2$, $\epsilon_3 = \varepsilon_3$, $\epsilon_4 = -\varepsilon_4$, and $\epsilon_6 = \varepsilon_6$. 
The factor $\epsilon_5$ deals with the over-counting due to simultaneously rotating the asymptotic markers of $u$ and $u'$.
This completes the proof of the claim.
\end{proof}

\begin{corollary}[Contact homology \cite{pardon2019contact, bao2015semi, ishikawa2018construction}]
Let $\mathfrak A$ be the differential graded commutative algebra generated freely by all good Reeb orbits over $\Q[H_2(M)]$. 
Let $\partial: \mathfrak A \to \mathfrak A$ be the differential defined over generators by 
\begin{equation} \label{eqn: partial}
\partial \gamma_+=\sum_{[\bs\gamma_-]}\sum_{A\in H_2(M;\Z)} \frac{1}{m(\bs \gamma_-)} | \mathcal M^0 (\gamma_+;\bs \gamma_-;A) | \cdot  e^A    \gamma_{-,1} \dots \gamma_{-,k},
\end{equation}
where $\bs \gamma_- = \gamma_{-,1}\dots \gamma_{-,k}$, and $m(\bs \gamma_-) = m(\gamma_{-,1}) \dots m(\gamma_{-,k})$.
Then $\partial^2 = 0$.
\end{corollary}
\begin{proof}
This follows from Theorem~\ref{theorem: main} by restricting to the term $g = 0$ and linear  $p_{\bs\gamma_+^\dagger}$.
\end{proof}
\begin{remark}
If one uses $o_{bm}$, then $\partial$ should be defined as 
\begin{equation} \label{eqn: partial}
\partial \gamma_+=\sum_{[\bs\gamma_-]}\sum_{A\in H_2(M;\Z)} \frac{1}{m(\bs \gamma_-)} | \mathcal M^0 (\gamma_+;\bs \gamma_-;A) | \cdot  e^A    \gamma_{-,k} \dots \gamma_{-,1}
\end{equation}
as in \cite{bao2015semi}.
\end{remark}

Lastly, we mention the follow result:
\begin{proposition}
Different choices of capping CR tuples $\T^\pm _S$ and the capping orientations $o(\T^\pm _S)$ produce isomorphic SFT.
\end{proposition}
\begin{proof}
Let ${\T'_S}^\pm$ and $o'({\T'_S}^\pm)$ be a different choices of capping CR tuples and their capping orientation. Let $\epsilon_S \in \{\pm 1\}$ be defined by $\epsilon_S o(\T^+_S) \sharp o'({\T'_S}^-) = \ocan( \T^+_S \sharp {\T'_S}^-)$. Then the isomorphism from $(\mathfrak W, D) \to (\mathfrak W, D')$ is defined by sending generators $p_\gamma, q_\gamma \mapsto \epsilon_S p_\gamma, \epsilon_S q_\gamma$ for all $\gamma$, where $S$ is the loop of symmetric matrices associated to $\gamma$. 
\end{proof}

\section*{Acknowlegements}
We thank Ko Honda, Russell Avdek, and Fan Zheng for helpful discussions.
\printbibliography
\end{document}